\documentclass{article}
\usepackage[utf8]{inputenc}
\usepackage{amsmath, amssymb, enumerate, multicol, chngcntr}
\usepackage[exercise]{amsthm}
\usepackage{xcolor}
\usepackage{pdfpages}
\usepackage{comment}
\usepackage{caption}
\usepackage{bm}
\usepackage{mathtools}

\theoremstyle{plain}
\newtheorem{thm}{Theorem}[section]

\newtheorem{lemm}[thm]{Lemma}

\newcommand{\lcm}{{\rm lcm}}

\counterwithin{equation}{section}

\usepackage{fancyhdr}
\usepackage{setspace}

\usepackage[english]{babel}
\usepackage[euler]{textgreek}
\usepackage{stmaryrd}

\usepackage{tikz}
\usetikzlibrary{shapes,arrows}
\usepackage{tikz-cd}
\usepackage{amssymb}
\usepackage{graphicx}
\graphicspath{{Images/}}
\usepackage[font=small,labelfont=bf]{caption}
\usepackage{pgf,tikz,pgfplots}
\pgfplotsset{compat=1.15}
\usepackage{mathrsfs}
\usepackage{diagbox}
\usepackage{float}
\usepackage{xcolor}
\usepackage{hyperref}
\usepackage{ragged2e}
\usepackage{enumerate}
\usepackage{upgreek}
\usepackage{multicol}

\theoremstyle{remark}

\theoremstyle{plain}

\newtheoremstyle{note}
  {3pt}
  {3pt}
  {}
  {}
  {\itshape}
  {:}
  {.5em}
  {}

\newtheorem{note}{Note}

\newtheoremstyle{citing}
  {3pt}
  {3pt}
  {\itshape}
  {}
  {\bfseries}
  {.}
  {.5em}
  {\thmnote{#3}}

\theoremstyle{citing}

\newtheoremstyle{break}
  {9pt}
  {9pt}
  {\itshape}
  {}
  {\bfseries}
  {.}
  {\newline}
  {}

\let\lvert=|\let\rvert=|

\title{Conjugacy classes of completely reducible cyclic subgroups of $\mbox{GL}(2, q)$}
\author{Prashun Kumar \footnote{Dr. B. R. Ambedkar University Delhi, Delhi 110006, India; \ E-mails: prashunkumar.19@stu.aud.ac.in,  prashun07kumar@gmail.com.} 
\ and \ Geetha Venkataraman\footnote{Corresponding Author, Dr. B. R. Ambedkar University Delhi, Delhi 110006, India; E-mails: geetha@aud.ac.in, geevenkat@gmail.com.}}

\begin{document}
\fontfamily{cmr}\selectfont

\maketitle


\bigskip
\noindent
{\small{\bf ABSTRACT:}}
 Let $m$ be a positive integer such that $p \nmid m$ where $p$ is prime. In this paper we find the number of conjugacy classes of completely reducible cyclic subgroups in $\mbox{GL}(2,q)$ of order $m$, where $q$ is a power of $p$.

\medskip
\noindent
{\small{\bf Keywords}{:} }
general linear group, conjugacy class, cyclic subgroup, completely reducible subgroup.

\medskip
\noindent
{\small{\bf Mathematics Subject Classification-MSC2020}{:} }
20E45, 20H30, 20K01, 20K25

\baselineskip=\normalbaselineskip

\vspace{.25in}
\noindent
{\bf THIS IS AN EARLY VERSION OF THE PAPER. FOR THE FINAL VERSION SEE DOI: 10.1080/00927872.2023.2179634.}
\vspace{.25in}

\section{Introduction}

Conjugacy classes of special subgroups of groups have been studied over the years. These have not only been researched independently but also because they are useful in answering other questions related to groups. Some of the papers in the former category are \cite{M2008}, \cite{B2009}, \cite{LMS2005}, \cite{B2021}, \cite{BFG1995}. Order and bound on the number of conjugacy classes of maximal solvable subgroups of the symmetric and general linear groups have also played a role in enumeration of groups, see \cite[Ch 11 and Ch 14]{BNV2007}. 

In particular, the conjugacy classes of cyclic subgroups of the general linear group have been studied in relation with a family of codes applicable for communications on a random linear network coding channel, see \cite{MTR2011}. A formula for the number of conjugacy classes of reducible subgroups of $\mbox{GL}(2,t)$ of orders $p, p^2, pr$ where $p,r$ and $t$ are distinct primes has been given in \cite{DEP2022}. 

Note that there exist irreducible cyclic subgroups of order $m$ in $\mbox{GL}(n,q)$ if and only if $m \mid q^2-1$ and $m \nmid q-1$. Further, if there exist irreducible cyclic subgroups of a given order $m$ in $\mbox{GL}(n,q)$ then they form a single conjugacy class, see \cite[Theorem 2.3.3]{S1992}. In our paper, we find the conjugacy classes of completely reducible cyclic subgroups of $\mbox{GL}(2,q)$ which are not irreducible and provide a formula for the number of conjugacy classes of such subgroups when $q$ is a power of a prime. By the above discussion, if $m \mid q-1$ then any cyclic subgroup of order $m$ will be reducible. The main result we prove is given below. 

\smallskip 
\noindent
{\bf Theorem A}\label{conj_clses_of_cylc_sbgrps}
{\em Let $q = p^k$ where $p$ is a prime. Let $N(m)$ be the number of conjugacy classes of reducible cyclic subgroups of ${\rm GL}(2,q)$ of order $m$ where $m \mid q-1$. Let $m= p_0^{\beta_0}p_1^{\beta_1} \ldots p_r^{\beta_r}$ be the prime decomposition for $m$ where the $p_i$'s are odd primes for $i \geq 1$ and $p_0=2$.  Then
$$N(m) = \frac{1}{2}(\rho(m) + \delta(m)) $$

where $ \rho(m) = \displaystyle\prod_{i=0}^{r}(p_i^{\beta_i} + p_i ^{\beta_i -1})$  and $ \delta(m) =
\begin{cases}
 2^{r}\  & \quad if \quad 0 \leq \beta_0 \leq 1\\
 2^{r+1} &  \quad if \quad \beta_0 = 2 \\
 2^{r+2} &  \quad if \quad \beta_0 \geq 3.
\end{cases}
$ \\
If $m$ is odd then we let $ \rho(m) = \displaystyle\prod_{i=1}^{r}(p_i^{\beta_i} + p_i ^{\beta_i -1})$.}

\bigskip
Let $m$ be as above. Let ${\cal X} = \{ [H] \mid H \leq \mbox{GL}(2,q), H \cong {\mathbb{Z}}_m \mbox{ and } H \mbox{ is reducible} \}$. Then with the notation of Theorem A, we have that $N(m) = |{\cal X}|$.

\medskip
The paper is organised as follows. We prove Theorem A in Section \ref{conj_class_rcyc} and finally in Section \ref{Misc} we provide some additional comments and remarks. Throughout the paper, $p$ is a prime, $q$ is a power of $p$ and $\mathbb{F}_q$ is the finite field of order $q$.

\smallskip

Let $D(2,q)$, denote the subgroup of diagonal matrices of $\mbox{GL}(2,q)$. Any $d \in D(2,q)$ with diagonal entries $d_1$ and $d_2$ will be represented as $dia(d_1,d_2)$. Let $M(2,q) = D(2,q) \rtimes S_2$ be the group of all monomial matrices of $\mbox{GL}(2,q)$, where $S_n$ denotes the symmetric group of degree $n$. 
\smallskip
\noindent
Let $m$ be such that $m \mid q-1$. Let ${\cal Y} = \{[H]_M \mid H \leq D(2,q) \mbox{ and } H \cong {\mathbb{Z}}_m \}$ where $[H]_M$ denotes the conjugacy class of $H$ in $M(2,q)$. The aim is to establish a one-to-one correspondence between ${\cal X}$ and ${\cal Y}$. \\

 Lemmas  \ref{com_red_subgrp} and \ref{red_cylc_subgrps} are standard results and can be proven easily. 
 
 \begin{lemm}\label{com_red_subgrp}
Let $G\leq {\rm GL}(2,q)$ be a subgroup with $p \nmid |G|$. Further let $V = {(\mathbb{F}_q)}^2$ be the natural $\mathbb{F}_qG$-module. If $G$ is reducible then $G$ conjugates into the group $D(2,q)$ of all diagonal matrices. Further  if $G$ is cyclic of order $m$ then $m \mid q-1$.

\end{lemm}

\begin{lemm}
\label{red_cylc_subgrps}
Let $m$ be a positive integer such that $m \mid q-1$. Then there exists a reducible cyclic subgroup $H \leq {\rm GL}(2,q)$ of order $m$. Further, $D(2,q)$ has a unique subgroup $U \cong \mathbb{Z}_m \times \mathbb{Z}_m$ such that every cyclic subgroup in $D(2,q)$ of order $m$ is contained in $U$.
\end{lemm}

\begin{lemm}\label{sbgrps_of_mono_matx}
Let $M(2,q) = D(2,q) \rtimes \langle a \rangle$ be the subgroup of monomial matrices in ${\rm GL}(2,q)$, where $a =$ $\begin{pmatrix}
    0 & 1\\
    1 & 0
\end{pmatrix}$. If $H_1,H_2 \leq D(2,q)$ are two distinct subgroups such that $H_2 = g{H_1}g^{-1}$ for some $g\in {\rm GL}(2,q)$, then $g \in M(2,q)$.
\end{lemm}

\begin{proof} Clearly $H_1$ is not contained in the center of ${\rm GL}(2,q)$ and hence it contains a non-scalar matrix $h_1$. Note that the anti-diagonal entries of  $g{h_1}g^{-1}$ have to be zero. Direct calculation then gives $g \in M(2,q)$.

\end{proof}

\section{Conjugacy classes of reducible cyclic subgroups of ${\rm GL}(2,q)$ whose order is coprime to $p$}
\label{conj_class_rcyc}
We shall prove Theorem A in this section. As stated earlier the aim of Theorem A is to find the size of ${\cal X}$. Clearly ${\rm GL}(2, q)$ acts by conjugation on all its cyclic subgroups of order $m$ where $m \mid q-1$ and ${N(m)} =|{\cal X}|$ is just the number of orbits under this action. We will use Burnside's Lemma \cite[Theorem 3.22, page 58]{JR1995} for counting the required orbits. However we prove a few standard results that will enable the use of Burnside's Lemma.

\begin{lemm}\label{elmts_of_odr_2}
Let $m = 2^{\beta_0}p_1^{\beta_1} \ldots p_r^{\beta_r}$ where $\beta_0 \geq 0$, $\beta_i \geq 1$ for $1\leq i \leq r$ and $p_i$ for $1 \leq i \leq r$ are odd primes then the number of elements of order 2 in $ \rm{Aut}(\mathbb{Z}_m)$, the automorphism group of $\mathbb{Z}_m$ is 
\[ 
\begin{cases}
  2^r -1 & {\rm when \ } 0 \leq \beta_0 \leq 1,\\
  2^{r+1} -1 & {\rm when \ } \beta_0 =2, \\
  2^{r+2} -1 & {\rm when \ } \beta_0 \geq 3.
\end{cases}
\]
\end{lemm}

\begin{proof} Suppose $P = P_1 \times P_2 \times \ldots \times P_k$ where $P_i$'s are cyclic groups. Then it can be seen that there are exactly $2^k -1$ choices of such $g$ and by \cite[Theorem 7.3, page 157]{JR1995} we get the desired formula.
\end{proof}

\begin{note}\label{spcl_elmnts_of_odr_2}
Let $m = 2^{\beta_0}$ where $\beta_0 \geq 1$. Then $\rm{Aut}(\mathbb{Z}_m)$ is trivial for $\beta_0 =1$, cyclic of order $2$ for $\beta_0 =2$ and isomorphic to $\mathbb{Z}_{2^{\beta_0-2}} \times \mathbb{Z}_2$. Thus the number of elements of order $2$ in $\rm{Aut}(\mathbb{Z}_m)$ is

\[ 
\begin{cases}
   0   &      {\rm when \ } \beta_0 = 1,\\
  1    &       {\rm when \ } \beta_0 =2, \\
   3 &       {\rm when \ } \beta_0 \geq 3.
\end{cases}
\]

\end{note}
For the result given below we shall follow the notation established in Lemma \ref{sbgrps_of_mono_matx}.

\begin{lemm}
\label{elmts_of_Fix(da)}
Let $m$ be a positive integer such that $m \mid q-1$. Let 
$$\hat{\cal Y} = \{ H \leq D(2,q) \mid H \cong {\mathbb{Z}}_m  \}.$$ Let ${\rm Fix}(g)$ denote the set of all elements of $\hat{\cal Y}$, fixed by $g \in M(2, q)$, under the conjugation action of $M(2, q)$ on $\hat{\cal Y}$. Let $d \in D(2,q)$ and $H \in \hat{\cal Y}$. 

\begin{enumerate}[$(a)$]
    \item $H \in {\rm Fix}(a)$ if and only if $H = \langle dia(\lambda, \lambda^l) \rangle$ where $\lambda \in {\mathbb{F}_q}^{*}$ with $o(\lambda)=m$ and $l \in \rm{Aut}(\mathbb{Z}_m)$ with $l^2= 1$.
    \item Let $H_1 = \langle dia(\lambda_1, {\lambda_1}^{l_1}) \rangle$ and $H_2 = \langle dia(\lambda_2, {\lambda_2}^{l_2}) \rangle$ where $\lambda_i \in {\mathbb{F}_q}^{*}$ and $l_i  \in \rm{Aut}(\mathbb{Z}_m)$ satisfy $o(\lambda_i) = m$ and $l_i^{2} = 1$ for $i = 1, 2$. Then $H_1 = H_2$ if and only if $ l_1 = l_2$.
\end{enumerate}

\end{lemm}

\begin{proof}

Let $H = \langle dia(\lambda,\lambda') \rangle$ be in ${\rm Fix}(a)$. Since $aHa^{-1}=H$ we get $dia(\lambda',\lambda) = {dia(\lambda,\lambda')}^l$ for some $l \in \rm{Aut}(\mathbb{Z}_m)$. So $\lambda = {(\lambda')}^l$  and $\lambda' = \lambda^l$. Using this we get that $m = \lcm (o(\lambda),o(\lambda')) \mid l^2-1$. Hence $dia(\lambda,\lambda') = dia(\lambda,{\lambda}^l)$. So $o(\lambda) = m$ and $l^2= 1$. 

Conversely if $H = \langle dia(\lambda, \lambda^l) \rangle$ where $\lambda \in {\mathbb{F}_q}^{*}$ with $o(\lambda)=m$ and $l \in \rm{Aut}(\mathbb{Z}_m)$ with $l^2 = 1\bmod{m}$, then $H \cong \mathbb{Z}_m$ and $aha^{-1} = h^l$ where $h = dia(\lambda, \lambda^l)$. Hence $H \in {\rm Fix}(da)$.

\smallskip
\noindent 
If $H_1 = H_2$, then $dia(\lambda_1, {\lambda_1}^{l_1}) = {dia(\lambda_2, {\lambda_2}^{l_2})}^{j}$ for some $j \in \rm{Aut}(\mathbb{Z}_m)$. Using this it is not difficult to show that $l_1=l_2$.
Now suppose $l_1=l_2$, then since $\lambda_2 \in \langle \lambda_1 \rangle $ we have $\lambda_2 = {\lambda_1}^{j}$ for some $j \in \rm{Aut}(\mathbb{Z}_m)$ which implies $dia(\lambda_2, {\lambda_2}^{l_2}) \in H_1$. Hence $H_1 = H_2$. 

\end{proof}

In the result below, we continue with the notation established in the above lemma. While we do not require Lemma \ref{Fix_da} for the proof of Theorem A, it is of independent interest.

\begin{lemm}\label{Fix_da}
  Let $d \in D(2, q)$. The function from ${\rm Fix}(a)$ to $\rm{Aut}(\mathbb{Z}_m)$ defined by $\langle dia(\lambda,{\lambda}^{l}) \rangle \longmapsto l$ is bijective if and only if $m = 2^{t}3^{s}$ where $0\leq t \leq 3$ and $0 \leq s \leq 1$.

\end{lemm}

\begin{proof} Part (a) of Lemma \ref{elmts_of_Fix(da)} gives us that
$${\rm Fix}(a) = \{ \langle dia(\lambda,{\lambda}^{l}) \rangle \mid \lambda \in {\mathbb{F}_q}^{*}, o(\lambda) = m, \ l \in \rm{Aut}(\mathbb{Z}_m) \mbox{ and } l^2 = 1 \}.$$ Using Part (b) and above we do know that $\langle dia(\lambda,{\lambda}^{l}) \rangle \longmapsto l$ defines a function from ${\rm Fix}(da)$ to $\rm{Aut}(\mathbb{Z}_m)$ which is injective. Clearly the above function will be onto if and only if every element of $\rm{Aut}(\mathbb{Z}_m)$ satisfies $l^2 = 1$. 

Let $p$ be an odd prime and let $\alpha \geq 2$. If $p^{\alpha}| m$ and $p^{\alpha +1} \nmid m$   then $\rm{Aut}(\mathbb{Z}_{p^\alpha}) \leq \rm{Aut}(\mathbb{Z}_m)$ and $\rm{Aut}(\mathbb{Z}_{p^\alpha})$ contains an element of order $p$. Now let us assume that $m = 2^tp$ with $t \geq 3$. Then $\rm{Aut}(\mathbb{Z}_p) \cong \mathbb{Z}_{p-1}$ and $\rm{Aut}(\mathbb{Z}_{2^t}) \cong \mathbb{Z}_2\times \mathbb{Z}_{2^{t -2}}$ are subgroups of $\rm{Aut}(\mathbb{Z}_m)$. So if $p > 3$ or $t > 3$, then these subgroups contain elements of order greater than $2$. So if every element of $\rm{Aut}(\mathbb{Z}_m)$ satisfies $l^2 = 1$ then $m = 2^{t}3^{s}$ where $0\leq t \leq 3$ and $0 \leq s \leq 1$.

\smallskip
\noindent
Conversely let us assume that $m = 2^{t}3^{s}$ where $0\leq t \leq 3$ and $0 \leq s \leq 1$. Then
\[ \rm{Aut}(\mathbb{Z}_m) \cong
\begin{cases}
    1   & s = 0, \, 0 \leq t \leq 1;\\
    \mathbb{Z}_2 & s = 0, \, t = 2 \quad \rm{or} \quad s = 1, \,   0 \leq t \leq 1;\\
    \mathbb{Z}_2\times \mathbb{Z}_2 & s = 0, \, t = 3 \quad \rm{or} \quad s = 1, \, t = 2;\\
    \mathbb{Z}_2\times \mathbb{Z}_2 \times \mathbb{Z}_2 & s = 1, \, t = 3.
\end{cases}
\]
\end{proof} 

We will now present the proof of Theorem A. The aim will be to establish a one-to-one correspondence between the sets ${\cal X}$ and ${\cal Y}$ discussed earlier.

\medskip
\noindent
{\bf Proof of Theorem A} 

\medskip
\noindent
Let $m$ be a positive integer such that $m \mid q-1$ where $q = p^k$ and $p$ is prime. Let the prime decomposition for $m= p_0^{\beta_0}p_1^{\beta_1} \ldots p_r^{\beta_r}$ where the $p_i$'s are odd primes for $i \geq 1$ and $p_0=2$. We assume that $\beta_i \geq 1$ for $i \geq 1$, $\beta_0 \geq 0$ and that $m \mid q-1$.
\smallskip

Recall that ${\cal X} = \{ [H] \mid H \leq \mbox{GL}(2,q), H \cong {\mathbb{Z}}_m \mbox{ and } H \mbox{ is reducible} \}$ and $N(m) = |{\cal X}|$. We had also defined ${\cal Y} = \{[K]_M \mid K \leq D(2,q) \mbox{ and } K \cong {\mathbb{Z}}_m \}$. Using Lemma \ref{com_red_subgrp}, we get that any reducible subgroup of ${\rm GL}(2, q)$ conjugates to a subgroup of $D(2, q)$. So in particular we note that for any $H \leq \mbox{GL}(2,q)$ such that $[H] \in {\cal X}$, there exists a $\hat{H} \leq D(2,q)$ such that $[{\hat H}]_M \in {\cal Y}$ with ${\hat H} \in [H]$.
\smallskip

Further, Lemma \ref{sbgrps_of_mono_matx}, shows that if two distinct subgroups of $D(2,q)$ are conjugated by an element of $\rm{GL}(2,q)$ then they are conjugated by an element of $M(2,q)$. Thus the map from ${\cal X}$ to ${\cal Y}$ given by $[H] \longmapsto [{\hat H}]_M$ turns out to be bijective. Hence we can conclude that ${\cal N} =|{\cal X}| = |{\cal Y}|$.
\smallskip
 
Now we establish the value of ${N(m)}$ by finding $|{\cal Y}|$, which is the same as counting the orbits of ${\cal \hat{Y}}$ (introduced in Lemma \ref{elmts_of_Fix(da)}), under the conjugation action of $M(2,q)$.
\smallskip

By Lemma \ref{red_cylc_subgrps} there exists a unique subgroup $U$ of $D(2,q)$ such that $U \cong \mathbb{Z}_m \times \mathbb{Z}_m$ and such that any $K \leq D(2,q)$ which is cyclic of order $m$, is contained in $U$. Now, any cyclic group is clearly the direct product of its Sylow subgroups. Thus $|{\hat{\cal {Y}}}| = \prod^{r}_{i=0} t_i$ where $t_i$ is the number of cyclic subgroups of order $p_i^{\beta_i}$ in $U$ and provided $\beta_0 \geq 1$. If not, the product will start at $i=1$. Clearly such a cyclic subgroup of order  $p_i^{\beta_i}$ is generated by an element of the form $dia(\lambda_1,\lambda_2)$ where $\lambda_i \in \mathbb{Z}_{m} \leq {\mathbb{F}_{q}}^*$. Further the order of one of the $\lambda_i$ is $p_i^{\beta_i}$ and the order of the other divides $p_i^{\beta_i}$. Therefore
\begin{align*}
t_i & = \frac{{(\varphi(p_i^{\beta_i}))}^2 + 2\sum_{j=1}^{\beta_i} \varphi(p_i^{\beta_i})\varphi(p_i^{\beta_i-j})}{\varphi(p_i^{\beta_i})}\\
    & = \varphi(p_i^{\beta_i}) + 2 \{\varphi(p_i^{\beta_i -1}) + \ldots + \varphi(p_i) + 1\}\\
    & = p_i^{\beta_i} + p_i^{\beta_i - 1}
\end{align*}
\noindent
where $\varphi$ is the Euler's $\varphi$-function.  Hence ${|\hat{\cal {Y}}|} = \prod^{r}_{i=0}(p_i^{\beta_i} + p_i^{\beta_i -1})$ provided $\beta_0 \geq 1$. If not, the product will start from $i=1$. By \cite[Theorem 3.22, page 58]{JR1995}, the number of orbits required
\begin{equation}{\cal{N}}\label{eq: no_of_orbts}
 = \frac{1}{2|D(2,q)|} \, \left (\sum_{d\in D(2,q)}  | {\rm Fix}(d)| + \sum_{d\in D(2,q)} |{\rm Fix}(da)|\right) \tag{$*$}
\end{equation}
Clearly $d \in D(2,q)$ fixes every element of $\hat{\cal{Y}}$. Thus $|{\rm Fix}(d)| = |\hat{\cal{Y}}|$. Also by Lemma \ref{elmts_of_Fix(da)}, we get that $ |{\rm Fix}(da)| = 1$ + Number of elements of order $2$ in $\rm{Aut}(\mathbb{Z}_m)$. So using Lemma  \ref{elmts_of_odr_2} to get $|{\rm Fix}(da)|$ and putting these values in (\ref{eq: no_of_orbts}) we get the desired result for ${N(m)}$. 

If $m = 2^{\beta_0}$ where $\beta_0 \geq 1$, then $|{\hat{\cal Y}}| = t_o$ and hence $|{\hat{\cal Y}}| = 2^{\beta_0} + 2^{\beta_0 -1} $. Again by Lemma \ref{elmts_of_Fix(da)} we get that $|{\rm Fix}(da)| = 1 +$ Number of elements of order $2$ in ${\rm Aut}(\mathbb{Z}_m)$. Thus we can use Note \ref{spcl_elmnts_of_odr_2} to  get $|{\rm Fix}(da)|$ and putting these values in (\ref{eq: no_of_orbts}) we get the desired result for ${N(m)}$.

\section{Miscellanea}\label{Misc}
There are alternative methods for finding conjugacy classes for cyclic reducible subgroups of order $m$ of ${\rm GL}(2,q)$ for small values of $m$ when $m \mid q-1$. For example, if $m$ is small then one can look at the minimal polynomials of matrices of order $m$ and pick the reducible(diagonalisable) ones.

We can also describe the generators of the cyclic subgroups of order $m$ that are representatives of the conjugacy classes of the reducible cyclic subgroups of order $m$ where $m \mid q-1$. Note that the representatives can be chosen as subgroups of $D(2,q)$. This is discussed below.

Let $\lambda \in {\mathbb{F}_q}^*$ be an element of order $m$ where $m \geq 3$ and let $M = M(2,q)$.
Let $H = \langle dia(\lambda_1,\lambda_2) \rangle$ be a subgroup of $D(2,q)$ of order $m$. Then $N_{M}(H)$ is either $D(2, q)$ or $M(2,q)$. If $N_{M}(H) = M(2,q)$ then $H = \langle dia(\lambda, \lambda^k) \rangle$ with $k^2=1\mod{m}$. 

We say $H$ is of Type ${\rm I}$ if $k = 1\mod{m}$, of Type ${\rm II}$ if $k^2 = 1\mod{m}$ but $k \neq 1 \mod{m}$ and of Type ${\rm III}$ if $N_{M}(H) = D(2,q)$. 

Now let the prime decomposition for $m= p_0^{\beta_0}p_1^{\beta_1} \ldots p_r^{\beta_r}$ where the $p_i$'s are odd primes and $p_0=2$. We assume that $\beta_i \geq 0$ for all $i$ and that at least one $\beta_j \geq 1$. We also assume that $m \mid q-1$. Define $m_i = m/{p_i}^{\beta_i} $ and $\lambda_i = \lambda^{m_i} $ where $i \in \{0, \ldots ,r \} $. Then the subgroup $H_i$ of $D(2,q)$ of order $p_i^{\beta_i}$ is of the form $H_i = \langle dia(\lambda_i, \lambda_i^{k_i}) \rangle$ where $0 \leq k_i \leq p_i^{\beta_i} -1$. So if $H\leq D(2,q)$ is a subgroup of type ${\rm III}$, then $H = \prod_{i=0}^{r} H_i$.

We summarise the above as follows.  Let $m \mid q-1$ and let $\lambda \in {\mathbb{F}_q}^*$ be an element of order $m$ where $m \geq 3$. Then the generators of the representatives of the conjugacy classes of reducible cyclic subgroups of ${\rm GL}(2, q)$ can be classified into three types as given below.

\begin{enumerate}[(a)]
\item Type ${\rm I}$: $\begin{pmatrix} 
	\lambda & 0  \\
	0 & \lambda \\ 
	\end{pmatrix}
$. \\
\item Type ${\rm II}$: $\begin{pmatrix} 
	\lambda & 0  \\
	0 & \lambda^{k} \\ 
	\end{pmatrix} $  where $k^2 = 1\mod{m}$ but $k \neq 1 \mod{m}$. \\
\item Type ${\rm III}$: $ \begin{pmatrix} 
	\lambda_0\cdots \lambda_r & 0  \\
	0 & {\lambda_0}^{k_0}\cdots {\lambda_r}^{k_r} \\ 
	\end{pmatrix}$ where $\lambda_i = \lambda^{m_i}$, $m_i = m/{p_i}^{\beta_i} $ and \newline $0 \leq k_i \leq p_i^{\beta_i}-1$ for all $i \in \{0, \ldots ,r \} $. Note only those $\lambda_i$ will occur in the product for which $\beta_i \geq 1$. 
\end{enumerate}

For $m=2$, Type ${\rm I}$ and Type ${\rm II}$ are the same and we can choose $\lambda = -1$. So the generators for the two types of representatives can be taken as $dia(-1,-1)$ or $dia(-1, 1)$.

\section*{Acknowledgements}

We would like to thank the referees for their suggestions and comments related to this paper.  \\

\noindent Prashun Kumar would like to acknowledge the UGC-JRF grant ({\emph{identification number}}: 201610088501) which is enabling his doctoral work.

\end{document}